\documentclass[a4paper,8pt]{article}
\usepackage{amsmath,fancyhdr,amssymb,amsthm, geometry, enumerate, graphicx ,amsfonts,hyperref,mathrsfs, calrsfs}
\DeclareMathAlphabet{\pazocal}{OMS}{zplm}{m}{n}

\author{Radhika Vasisht $^{1,\dag}$ and Ruchi Das $^{1}$}
\title{Exploring $\mathcal{F}$-Sensitivity for Non-Autonomous Systems} 

\theoremstyle{definition}
\newtheorem{defn}{Definition}[section]
\providecommand{\keywords}[1]{\textbf{Keywords :} #1}
\providecommand{\msc}[1]{\textbf{MSC(2010)} #1}
\theoremstyle{plain}
\newtheorem{thm}{Theorem}[section]

\newtheorem{exm}{Example}[section]
\newtheorem{rmk}{Remark}[section]

\newtheorem{lem}{Lemma}[section]

\begin{document}
\date{}
\maketitle
\begin{abstract}
We study some stronger forms of sensitivity, namely, $\mathcal{F}$-sensitivity and weakly$\mathcal{F}$-sensitivity for non-autonomous discrete dynamical systems. We obtain a condition under which these two forms of sensitivity are equivalent. We also justify the difference between $\mathcal{F}$-sensitivity and some other stronger forms of sensitivity through examples. We explore the relation between the $\mathcal{F}$-sensitivity of the non-autonomous system $(X,f_{1,\infty})$ and autonomous system  $(X,f)$, where $f_n$ is a sequence of continuous functions converging uniformly to $f$. We also study the $\mathcal{F}$-sensitivity of a non-autonomous system $(X,f_{1,\infty})$, generated by a finite family of maps $\mathbb{F}=\{f_1, f_2, \ldots, f_k\}$ and give an example showing that such non-autonomous systems can be $\mathcal{F}$-sensitive, even when none of the maps in the family $\mathbb{F}$ is $\mathcal{F}$-sensitive.
\end{abstract}
\\ \msc{Primary 54H20; Secondary 37B55}

\renewcommand{\thefootnote}{\fnsymbol{footnote}}
\footnotetext{\hspace*{-5mm}
\renewcommand{\arraystretch}{1}
\begin{tabular}{@{}r@{}p{15cm}@{}}
$^\dag$& the corresponding author. Email address: radhika.vasisht92@gmail.com (R. Vasisht)\\
$^1$&Department of Mathematics, University of Delhi, Delhi-110007, India\\
\end{tabular}}

\section{Introduction}
Beginning with the contributions of Poincar\'e and Lyapunov, theory of dynamical systems has seen significant developments in recent years. This theory has gained considerable interest and has been found to have useful connections with many different areas of mathematics. An autonomous discrete dynamical system is a dynamical system which has no external input and always evolves according to the same unchanging law. Most of the natural systems in this world are subjected to time-dependent external forces and their modelling leads to a mathematical theory of what is called non-autonomous discrete dynamical system.  The theory of non-autonomous dynamical systems helps characterizing the behaviour of various natural phenomenons which cannot be modelled by autonomous systems. Over recent years, the theory of such systems has developed into a highly active field related to, yet recognizably distinct from that of classical autonomous dynamic systems \cite{MR2960260, MR3206430, MR3344142, MR3485444}. We first introduce some notations. Consider the following non-autonomous discrete dynamical system (N.D.S) $(X,f_{1,\infty})$:
\begin{equation} x_{n+1}=f_n(x_n) , n \geq 1,\nonumber \end{equation}
where $(X,d)$ is a compact metric space and $f_n: X \rightarrow X$ is a continuous map. For convenience, we denote $({f_n})_{n =1}^\infty $ by $f_{1,\infty}$. Naturally, a difference equation of the form $x_{n+1}=f_n(x_n)$ can be thought of as the discrete analogue of a non-autonomous differential equation $\frac{dx}{dt}=f(x,t)$. 
\\ Chaos is a universal dynamical behavior of nonlinear dynamical systems and one of the central topics of research in nonlinear science. It is well known that sensitivity characterizes the unpredictability of chaotic phenomena, and is one of the essential conditions in various definitions of chaos. Therefore, the study on sensitivity has attracted a lot of attention from many researchers. For continuous self maps of compact metric spaces, Moothathu \cite{MR2351026} gave an insight of some stronger forms of sensitivity based on the largeness of subsets of $\mathbb{N}$. Since then several other stronger forms of sensitivity have been studied by different researchers for both autonomous and non-autonomous systems\cite{MR2922208, MR3528201}.
\\ In \cite{MR2600324}, the authors introduced and studied the concept of sensitivity via Furstenberg families for autonomous systems. A lot of work has been done in this direction since then \cite{MR3339062, MR3543466}. Recently, some researchers have introduced and studied various versions of sensitivity via Furstenberg families for non-autonomous discrete systems \cite{MR3623040, MR3584037}. Motivated by all the research done for sensitivity via Furstenberg families, we study $\mathcal{F}$-sensitivity and weakly $\mathcal{F}$-sensitivity for non-autonomous discrete dynamical systems.
\\In section 2, we recall some concepts about Furstenberg families and also give preliminaries required for other sections.
In section 3, we study the relations between $\mathcal{F}$-sensitivity and some other stronger forms of sensitivity . We obtain a condition under which $\mathcal{F}$-sensitivity and weakly $\mathcal{F}$-sensitivity for the system $(X,f_{1,\infty})$ are equivalent. We give some characterizations of $\mathcal{F}$-sensitivity. In section 4, we investigate $\mathcal{F}$-sensitivity and weakly $\mathcal{F}$-sensitivity for different non-autonomous systems. We explore the relation between the $\mathcal{F}$-sensitivity of the non-autonomous system $(X,f_{1,\infty})$ and autonomous system  $(X,f)$, where $f_n$ is a sequence of continuous functions converging uniformly to $f$. We also study the $\mathcal{F}$-sensitivity of a non-autonomous system $(X,f_{1,\infty})$, generated by a finite family of maps $\mathbb{F}=\{f_1, f_2, \ldots, f_k\}$ and give an example showing that such non-autonomous systems can be $\mathcal{F}$-sensitive, even when none of the maps in the family $\mathbb{F}$ is $\mathcal{F}$-sensitive.

\section{Preliminaries}

Let $(X,f_{1,\infty})$ be a non-autonomous dynamical system. For any two open sets U and V of $X$, let$N_{f_{1,\infty}}(U,V)= \{n\in \mathbb{N} : f_1^n(U)\cap V\ne \emptyset \}$. In particular, $N_{f_{1,\infty}}(x,V)= \{n\in \mathbb{N} : f_1^n(x)\in V\}$, where $f_1^i(x) = f_{i}\circ\cdots\circ f_1(x)$.

\begin{defn}
A system $(X,f_{1,\infty})$ is said to $\textit{feeble open}$ if for any non-empty open set $U$ in $X$, $int(f_{n}(U))$ is non-empty for each $n\in\mathbb{N}$.
\end{defn}

\begin{defn}
The system $(X,f_{1,\infty})$ is said to have $\textit {sensitive dependence on initial}$ $ \textit{conditions}$ if there exists a constant $ \delta_0>0$ such that for any $x_0 \in X$ and any neighbourhood U of $x_0$, there exists $y_0\in X\cap U$ and a positive integer n such that $d(x_n,y_n)>\delta_0$, where $\{x_i\}_{i=0}^\infty$ and  $\{y_i\}_{i=0}^\infty$ are the orbits of the system $(X,f_{1,\infty})$ starting from $x_0$ and $y_0$ respectively, the constant $\delta_0>0$ is called a sensitivity constant of the system $(X,f_{1,\infty})$.
Here  $({x_i})_{i=0}^\infty = \{x\in X$ such that $f_1^i(x) , i\geq 1\}$.
Thus, system $(X,f_{1,\infty})$ has sensitive dependence on initial conditions or is $\textit{sensitive}$ in $X$ if there exists a constant $\delta>0$ such that for any non-empty open set V of $X$, $N_{f_{1,\infty}}(V,\delta)$ is non-empty, where $N_{f_{1,\infty}}(V,\delta) = \{n\in \mathbb{N}$ such that, there exist   $x,y \in V$ with $ d(f_1^n(x),f_1^n(y))> \delta \}$.
\end{defn}

\begin{defn}
The system $(X,f_{1,\infty})$ is called $\textit{cofinitely sensitive}$ in $X$ if there exists a constant $\delta>0$ such that for any non-empty open set V of $X$,there exists $N\geq1$ such that $[N,\infty)\cap\mathbb{N} \subset N_{f_{1,\infty}}(V,\delta)$; $\delta$ is called a constant of cofinite sensitivity.
\end{defn}
  
\begin{defn}
The system $(X,f_{1,\infty})$ is said to have $\textit{syndetic sensitivity}$ in $X$ if there exists a constant $\delta>0$ such that for any non-empty open set V of $X$, $N_{f_{1,\infty}}(V,\delta)$ is syndetic; $\delta$ is called a constant of syndetic sensitivity.
\end{defn}
Now, we recall some concepts related to Furstenberg families.
Let $\mathcal{P}$ be the collection of all subsets of $\mathbb{Z^+}$. A collection $\mathcal{F}\subseteq \mathcal{P}$ is called a $\mathit{Furstenberg\hspace{0.15cm} family}$, if it is hereditary upwards, that is, $F_1 \subset F_2$ and $F_1\in \mathcal{F}$ implies $F_2\in \mathcal{F}$. A family $\mathcal{F}$ is proper if it is a proper subset of $\mathcal{P}$, that is, it is neither whole $\mathcal{P}$ nor empty. Throughout this paper, all Furstenberg families are proper.
\\For a Furstenberg family $\mathcal{F}$, the dual family\begin{align*}
\mathit{k} \mathcal{F} = \{F\in \mathcal{P}:F\cap F'\neq  \emptyset, for\hspace{0.15cm} all\hspace{0.15cm} F'\in \mathcal{F}\}=\{F\in \mathcal{P}: \mathbb{Z^+}\setminus F\notin \mathcal{F}\}
\end{align*} 
Clearly, if $\mathcal{F}$ is a Furstenberg family, then so is $\mathit{k}\mathcal{F}$. One can note that $\mathit{k}(\mathit{k}\mathcal{F})=\mathcal{F}$. Note that the family $\mathcal{B}$ of all infinite subsets of $\mathbb{Z^+}$ is a Furstenberg family and $\mathit{k}\mathcal{B}$ is the family of all cofinite subsets of $\mathbb{Z^+}$.
\\ For Furstenberg families $\mathcal{F}_1$ and $\mathcal{F}_2$, let $\mathcal{F}_1.\mathcal{F}_2=\{F_1 \cap F_2: F_1 \in \mathcal{F}_1, F_2 \in \mathcal{F}_2\}$. A Furstenberg family $\mathcal{F}$ is said to be filterdual if $\mathcal{F}$ is proper and $\mathit{k}\mathcal{F}.\mathit{k}\mathcal{F}\subseteq \mathit{k}\mathcal{F}$. A Furstenberg family $\mathcal{F}$ is said to be translation invariant if for any $F\in \mathcal{F}$ and any $i\in \mathbb{Z^+}, \hspace{0.15cm} F+i \in \mathcal{F}$ and $F-i\in \mathcal{F}$.
\\Let $A\subseteq X$ and $x\in X$. If $N_{f_{1,\infty}}(x,A)\in \mathcal{F}$, then $x$ is called an $\mathcal{F}$-attaching point of $A$. The set of all $\mathcal{F}$-attaching points of $A$ is called the $\mathcal{F}$-attaching set of $A$ and is denoted by $\mathcal{F}(A)$. Clearly, \begin{align*}
\mathcal{F}(A)= \bigcup_{F\in\mathcal{F}}\bigcap_{n\in\mathcal{F}}f_1^n(A).
\end{align*}
Let $(X,f_{1,\infty})$ be a N.D.S. A Furstenberg family $\mathcal{F}$ is compatible with the system $(X,f_{1,\infty})$ if the $\mathcal{F}$-attaching set of $U$ is a $G_{\delta}$ set of $X$, for each open set $U$ of $X$.
\\Throughout this paper, we will use the following relations on $X$. For a given $\delta>0$:
\begin{align*}
 V_{\delta}=\{(x,y): d(x,y)< \delta\},
\end{align*}
\begin{align*}
\overline{V_{\delta}} = \{(x,y): d(x,y)\leq \delta \}
\end{align*}

\begin{defn}
The set of $\mathcal{F}-\textit{asymptotic pairs}$, $Asym_{\delta}(\mathcal{F})=\{(x,y): N_{f_{1,\infty}}((x,y),\overline{V_{\delta}}\in \mathit{k}\mathcal{F})\}= \mathit{k}\mathcal{F}(\overline{V_{\delta}})$ and $Asym_{\delta}(\mathcal{F})(x)=\{y: (x,y)\in Asym_{\delta}(\mathcal{F})\}$.
\\Clearly, $Asym(\mathcal{F})=\bigcap_{\delta>0}Asym_{\delta}(\mathcal{F})$ and $Asym(\mathcal{F})(x)=\bigcap_{\delta>0}Asym_{\delta}(\mathcal{F})(x)$.
\end{defn}

\begin{defn}
The system $(X,f_{1,\infty})$ is $\textit{weakly}\hspace{0.15cm} \mathcal{F}-\textit{sensitive}$ if there exists a positive $\delta>0$, such that for every open subset $U$ of $X$, there exists $x,y \in U$ such that the pair $(x,y)$ is not $\mathcal{F}$-$\delta$-asymptotic, that is, $\{n\in \mathbb{N}: d(f_1^n(x),f_1^n(y))>\delta\}\in \mathcal{F}$ or we can say $N_{f_{1,\infty}}((x,y),X\times X\setminus \overline{V_{\delta}})\in \mathcal{F}$,  where $\delta$ is called a weakly $\mathcal{F}$-sensitive constant.
\end{defn}

\begin{defn}
The system $(X,f_{1,\infty})$ is $\mathcal{F}- \textit{sensitive}$ if there exists a positive $\delta>0$, such that for every $x\in X$ and every open neighbourhood $U$ of $x$, there exists $y\in U$ such that the pair $(x,y)$ is not $\mathcal{F}$-$\delta$-asymptotic, that is, the set $N_{f_{1,\infty}}(U,\delta)\in \mathcal{F}$,  where $\delta>0$ is an $\mathcal{F}$-sensitive constant. 
\end{defn}

Let $X$ be a topological space and $\pazocal{K}$($X)$ denote the hyperspace of all non-empty compact subsets of $X$ endowed with the \textit{Vietoris Topology}. A basis of open sets for Vietoris topology is given by following sets:
\bigskip

\noindent $< U_1, U_2, \ldots , U_k >$ = $\{K \in$ $\pazocal{K}$($X)$: $K \subset \bigcup_{i=1}^{k} U_{i}$ and $K\cap U_{i}$  $\ne \emptyset$, for each $i$ $\in \{1, 2, \ldots ,k\}$\},

\bigskip\noindent where $U_1, U_2, \ldots ,U_k$ are non-empty open subsets of $X$.

Given metric space $(X, d)$, a point $x \in X$ and $A \in \pazocal{K}(X)$, let $d(x, A)$ = $\inf \{d(x, a): a \in A \}$. For every $\epsilon > 0$, let  open $d$-ball in $X$ about $A$ and radius $\epsilon$ be given by  $B_{d}(A, \epsilon) = \{x \in X: d(x, A) < \epsilon\} = \bigcup_{a \in A} B_d(a, \epsilon)$, where $B_d(a, \epsilon)$ denotes the open ball in $X$ centred at $a$ and of radius $\epsilon$. The Hausdorff metric on $\pazocal{K}$($X)$  induced by $d$, denoted by $d_H$, is defined  as follows:

\ \ \ \ \ \ \ \ \ \ \ \ \ \ \ \ \ $d_H$($A, B) = \inf \{ \epsilon > 0: A \subseteq B_{d}(B, \epsilon) \  \text{and} \  B \subseteq B_{d}(A, \epsilon)\}$,

\bigskip\noindent where $A$, $B \in$ $\pazocal{K}$($X)$. We shall recall that the topology induced by the Hausdorff metric coincides with the Vietoris topology if and only if the space $X$ is compact. Also, for a compact metric space $X$ and  $A, B \in$ $\pazocal{K}$($X)$, we get that $d_H$($A, B) < \epsilon$ if and only if $A \subseteq B_d(B, \epsilon)$ and $B \subseteq B_d(A, \epsilon)$.
\\Let $\pazocal{F}(X)$ denote the set of all finite subsets of $X$. Under Vietoris topology, $\pazocal{F}(X)$ is dense in $\pazocal{K}(X)$ \cite{MR2665229, MR1269778}.
 Given a continuous function $f: X \to X$, it induces a continuous function $\overline{f}$: $\pazocal{K}$($X) \to$ $\pazocal{K}$($X)$ defined by $\overline{f}(K) = f(K)$, for every $K \in$ $\pazocal{K}$($X)$, where $f(K)$ = $\{f(k) : k\in K\}$. Note that continuity of $f$ implies continuity of $\overline{f}$.

\bigskip Let $(X, f_{1,\infty})$ be a non-autonomous discrete dynamical system and $\overline{f}_n$ be the function on $\pazocal{K}$($X)$, induced by $f_n$ on $X$, for every $n\in\mathbb{N}$. Then the sequence $\overline{f}_{1,\infty}$ = ($\overline{f}_1, \overline{f}_2$, $\ldots ,\overline{f}_n, \ldots )$ induces a non-autonomous discrete dynamical system ($\pazocal{K}$($X), \overline{f}_{1,\infty})$, where $\overline{f}_1^n = \overline{f}_n \circ \ldots \circ \overline{f}_2\circ \overline{f}_1$.  Note that $\overline{f}_1^n = \overline{f^n_1}$.
\\Let $(X,d_X)$ and $(Y,d_Y)$ be compact metric spaces. For non-autonomous discrete dynamical systems $(X,f_{1,\infty})$ and $(Y,g_{1,\infty})$, put $(f_{1,\infty}\times g_{1,\infty})= (h_{1,\infty}) = (h_1,h_2,\ldots,h_n,\ldots)$, where $h_n=f_n\times g_n$ , for each $n\in \mathbb{N}$. Thus, $(X\times Y, f_{1,\infty}\times g_{1,\infty})$ is a non-autonomous dynamical system, where $(X\times Y)$ is a compact metric space endowed with the product metric $d_{X\times Y}((x,y),(x',y'))= d_X(x,x')+d_Y(y,y')$. Here, $h_{1}^{n}=h_n\circ h_{n-1}\circ \cdots \circ h_2\circ h_1 = (f_n\times g_n)\circ (f_{n-1}\times g_{n-1})\circ \cdots \circ (f_2\times g_2)\circ (f_1\times g_1)$ \cite{MR3584171}.
\\Let $(X,d)$ be a compact metric space and let $C(X)$ denote the collection of continuous self-maps on $X$. For any $f,g\in C(X)$, the \textit{Supremum metric} is defined by \\$D(f,g)=\underset{x\in X}{sup}$ $d(f(x),g(x))$. It is easy to observe that a sequence $(f_n)$ in $C(X)$ converges to $f$ in $(C(X),D)$ if and only if $f_n$ converges to $f$ uniformly on $X$ and hence the topology generated by the \textit{Supremum metric} is called the topology of uniform convergence. 
\\ The $k$-th iterate is considered as $(X, f_{1,\infty}^{[k]})$, where $f_{1,\infty}^{[k]} = (f_{k(n-1)+ 1})_{n=1}^{\infty})$ and $f_{k(n-1)+ 1} = f_{kn-k+1+k-1}\circ f_{kn-1}\circ\cdots\circ f_{kn-k+1}$.
\\We recall the following results:
\begin{lem}[Corollary 2.2, \cite{MR3019971}]
Assume that the non-autonomous system $(X,f_{1,\infty})$ converges uniformly to a map $f$. Then for any $\epsilon>0$ and any $k\in \mathbb{N}$, there exists $\eta(\epsilon)>0$ and $N(k)\in \mathbb{N}$ such that for any pair $x,y \in X$ with $d(x,y)<\eta(\epsilon)$ and any $n\geq N(k)$, $d(f_n^k(x),f_n^k(y))<\epsilon/2$.
\end{lem}

\begin{lem}[Corollary 1, \cite{MR2}]
Let $(X,f_{1,\infty})$ be a N.D.S generated by a family $f_{1,\infty}$ and let $f$ be any continuous self map on $X$. If the family $f_{1,\infty}$ commutes with $f$ then for any $x\in X$ and any $k\in\mathbb{N}$, $d(f_{1}^{n+k}(x),f^{k}(f_{1}^{n}(x))) \leq \sum_{i=1}^k D(f_{i+1},f)$.
\end{lem}

\section{Relations of $\mathcal{F}$-sensitivity with some other stronger forms of sensitivity.}
In this section, we present our results about $\mathcal{F}$-sensitivity and weakly$\mathcal{F}$-sensitivity of the non-autonomous systems and discuss a condition under which these two are equivalent. We also give an example to show that $\mathcal{F}$-sensitivity need not imply syndetic sensitivity and cofinite sensitivity. 
\begin{thm}
Let $(X,f_{1,\infty})$ be a non-autonomous system and $\mathcal{F}$ be a filterdual. Then $(X,f_{1,\infty})$ is weakly $\mathcal{F}$-sensitive if and only if it is $\mathcal{F}$-sensitive.
\end{thm}

\begin{proof}
First suppose that $(X,f_{1,\infty})$ is $\mathcal{F}$-sensitive, with constant of $\mathcal{F}$-sensitivity $\delta>0$. Let $U$ be any open set in $X$. For any $x\in U$, since $U$ is an open neighbourhood of $x$, by $\mathcal{F}$-sensitivity of $(X,f_{1,\infty})$, there exists a $y\in U$ such that $(x,y)$ is not $\mathcal{F}-\delta$-asymptotic. Therefore, $U$ contains a pair $(x,y)$ which is not $\mathcal{F}-\delta$-asymptotic and hence $(X,f_{1,\infty})$ is weakly $\mathcal{F}$-sensitive.
\\Conversely, suppose that $(X,f_{1,\infty})$ is weakly $\mathcal{F}$-sensitive, with constant of weakly$\mathcal{F}$-sensitivity $\delta>0$. If $(X,f_{1,\infty})$ is not $\mathcal{F}$-sensitive, then for each $\delta>0$, there exist $x\in X$ and an open neighbourhood $U$ of $x$ such that $\{n\in\mathbb{N}: d(f_1^n(x),f_1^n(y))>\delta\} \notin \mathcal{F}$, for every $y\in X$. Therefore, $N_{f_{1,\infty}}((x,y),X\times X \setminus \overline{V_{\delta}})\notin \mathcal{F}$, for every $y\in X$ and hence
$\mathbb{N}\setminus N_{f_{1,\infty}}((x,y),\overline{V_{\delta}})\notin \mathcal{F}$ which implies $
N_{f_{1,\infty}}((x,y),\overline{V_{\delta}})\in \mathit{k}\mathcal{F}$.
Since $\mathcal{F}$ is filterdual, using triangle inequality, we get that $N_{f_{1,\infty}}((a,b),\overline{V_{2\delta}})\in \mathit{k}\mathcal{F}$. Therefore, $\mathbb{N}\setminus N_{f,\infty}((a,b),X\times X\setminus\overline{V_{2\delta}})\in \mathit{k}\mathcal{F}$ and hence $N_{f_{1,\infty}}((x,y),X\times X\setminus \overline{V_{2\delta}})\notin \mathcal{F}$ which contradicts that $(X,f_{1,\infty})$ is weakly $\mathcal{F}$-sensitive.
\end{proof}

\begin{rmk}
If $(X,f_{1,\infty})$ is $\mathcal{F}$-sensitive, then $(X,f_{1,\infty})$ is weakly $\mathcal{F}$-sensitive, even when $\mathcal{F}$ is not filterdual.
\end{rmk}

\begin{lem}
Let $(X,f_{1,\infty})$ be a N.D.S. A Furstenberg family $\mathcal{F}$ is compatible with system $(X,f_{1,\infty})$ if and only if the $\mathit{k}\mathcal{F}$-attaching set of V is a $F_{\sigma}$ set of $X$ for each closed subset $V$ of $X$. 
\end{lem}

\begin{proof}
Suppose $V$ is a closed subset of $X$, then $x\in\mathit{k}\mathcal{F}$
\begin{align*} & \iff N_{f_{1,\infty}}(x,V)\in \mathit{k}\mathcal{F} \\
& \iff \mathbb{N} \setminus N_{f_{1,\infty}}(x,X\setminus V) \in \mathit{k}\mathcal{F} \\
& \iff N_{f_{1,\infty}}(x,X\setminus V) \notin \mathcal{F} \\
& \iff x \notin \mathcal{F}(X\setminus V). \end{align*} Therefore, we have that $\mathit{k}\mathcal{F}(V)=
X\setminus \mathcal{F}(X\setminus V)$. Hence, $\mathit{k}\mathcal{F}(V)$ is an $F_{\sigma}$-set if and only if $\mathcal{F}(X\setminus V)$ is a $G_{\delta}$-set which is true by definition.

\end{proof}

The following result gives characterizations for the $\mathcal{F}$-sensitivity of non-autonomous system $(X,f_{1,\infty})$.
\begin{thm}
Let $(X,f_{1,\infty})$ be a N.D.S. If the Furstenberg family $\mathcal{F}$ is a filterdual and is compatible with the system $(X\times X, f_{1,\infty}\times f_{1,\infty})$, then the following are equivalent.
\begin{enumerate}
\item $(X,f_{1,\infty})$ is weakly $\mathcal{F}$-sensitive.
\item There exists a positive $\delta$ such that for every $x\in X$, $Asym_{\delta}(\mathcal{F})(x)$ is a first category subset of $X$.
\item There exists a positive $\delta$ such that $Asym_{\delta}(\mathcal{F})$ is a first category subset of $X\times X$.
\item There exists a positive $\delta$ such that $\mathcal{F}(X\times X \setminus \overline{V_{\delta}})$ is dense in $X\times X$.
\item $(X,f_{1,\infty})$ is $\mathcal{F}$-sensitive.
\end{enumerate}
\end{thm}

\begin{proof}
(1)$\implies$(2) We have that $\mathcal{F}$ is compatible with the system $(X\times X, f_{1,\infty}\times f_{1,\infty})$, so by Lemma 3.1, $Asym_{\delta}(\mathcal{F})$ is an $F_{\sigma}$ set of $X\times X$.

Suppose $Asym_{\delta}(\mathcal{F})=\bigcup_{i=1}^{\infty}A_n$, where every $A_i$ is a closed subset of $X\times X$, then $Asym_{\delta}(\mathcal{F})(x)=\bigcup_{i=1}^{\infty}A_n(x)$.
Suppose that for each $\delta>0$, there exists $x\in X$ such that $Asym_{\delta}(\mathcal{F})(x)$ is not of first category, then by Baire's Category Theorem there exists an open subset $U$ of $X$ such that $U\subset A_i(x)$, for some $i$. Hence, for each $y\in U$, $N_{f_{1,\infty}}((x,y),\overline{V_{\delta}})\in \mathit{k}\mathcal{F}$ and since $\mathcal{F}$ is filterdual, using triangle inequality we get that $N_{f_{1,\infty}}((a,b),\overline{V_{2\delta}})\in \mathit{k}\mathcal{F}$, for all $a,b \in U$. Then, $\mathbb{N}\setminus  N_{f_{1,\infty}}((a,b),X\times X\setminus  \overline{V_{2\delta}})\in \mathit{k}\mathcal{F}$. Therefore $N_{f_{1,\infty}}((a,b),X\times X\setminus \overline{V_{2\delta}})\notin \mathcal{F}$, which contradicts that $(X,f_{1,\infty})$ is weakly $\mathcal{F}$-sensitive.\\
(2)$\implies$(3) We know that $Asym_{\delta}(\mathcal{F})=\mathit{k}\mathcal{F}(\overline{V_{\delta}})$, therefore by Lemma 3.1, we get that $Asym_{\delta}({\mathcal{F}})$ is an $F_{\sigma}$ set. Suppose that $Asym_{\delta}(\mathcal{F})=\bigcup_{i=1}^{\infty}A_i$, where each $A_i$ is a closed subset of $X\times X$. Then, $Asym_{\delta}(\mathcal{F})(x)=\bigcup_{i=1}^{\infty}A_i(x)$. If $Asym_{\delta}(\mathcal{F})$ is not first category, then by Baire's Category Theorem, some $A_i$ has non-empty interior. If $U\times V \subset A_i$ and $x\in U$, then $V\subset A_i$. Hence, $Asym_{\delta}(\mathcal{F})(x)$ is not first category.\\
(3)$\implies$(1) If $(X,f_{1,\infty})$ is not weakly $\mathcal{F}$-sensitive, then for every $\delta>0$, there exists an open set $U$ in $X$ such that $N_{f_{1,\infty}}((x,y),X\times X \setminus \overline{V_{\delta}})\notin \mathcal{F}$, for each $(x,y)\in U\times U$, that is, $\mathbb{N}\setminus N_{f_{1,\infty}}((x,y),\overline{V_{\delta}})\notin \mathcal{F}$. Therefore, $N((x,y),\overline{V_{\delta}})\in \mathit{k}\mathcal{F}$ which implies $(x,y)\in Asym_{\delta}(\mathcal{F})$. Hence, $U\times U\subset Asym_{\delta}(\mathcal{F})$. Thus, $Asym_{\delta}(\mathcal{F})$ is not of first category.\\
\\Thus, we have proved that (1), (2) and (3) are equivalent.
\\(2)$\implies$(4) We first observe that $\mathcal{F}(X\times X\setminus \overline{V_{\delta}})=X\times X\setminus\mathit{k}\mathcal{F}(\overline{V_{\delta}})=X\times X\setminus Asym_{\delta}(\mathcal{F})$. Since, complement of a first category set is dense in a Baire's Space, we get the required result.
\\(4)$\implies$(1) Let $U$ be any open set in $X$. Since $\mathcal{F}(X\times X\setminus \overline{V_{\delta}})$ is dense in $X\times X$, therefore we have $U\times U\bigcap  \mathcal{F}(X\times X\setminus \overline{V_{\delta}})$ is non-empty which implies $U\times U\bigcap X\times X\setminus  Asym_{\delta}(\mathcal{F})$ is non-empty. Thus, there exists $(x,y)\in U\times U$ such that $(x,y)\notin Asym_{\delta}(\mathcal{F})$. Hence, $(X,f_{1,\infty})$is weakly $\mathcal{F}$-sensitive.
\\(1)$\iff$(5) follows directly from Theorem 3.1.
\end{proof}

The following example justifies that an $\mathcal{F}$-sensitive non-autonomous system need not be syndetic sensitive and cofinite sensitive.
\begin{exm} \end{exm}
Let $\mathbb{Z}$ denote the set of all integers, $\cal{A}$ = $\{0,1\}$ and $\Sigma_2$=$\cal{A}^\mathbb{Z}$ = $\{( \ldots, x_{-2}, x_{-1}$, \fbox{$x_0$}, $x_1, \ldots): x_i \in \{0,1\}, \ \text{for every} \  i  \in \mathbb{Z}\}$ with product metric  \[\rho(x,y) = \sum_{j=-\infty}^{\infty}\frac{|x_j-y_j|}{2^{|j|}} \]  
for any pair $ x= ( \ldots, x_{-2}, x_{-1}$, \fbox{$x_0$}, $ x_1, \ldots)$;  $y = ( \ldots, y_{-2}, y_{-1}$, \fbox{$y_0$}, $ y_1, \ldots) \in \Sigma_2$. The space $(\Sigma_2, \rho)$ is called the \emph{the two sided symbolic space}. Define $\sigma: \Sigma_2 \to \Sigma_2$ by $\sigma(x) = ( \ldots, x_{-2}, x_{-1}, x_0, \fbox{$x_1$}, x_2, \ldots)$, where $x = ( \ldots, x_{-2}, x_{-1}, $ \fbox{$x_0$}, $ x_1, x_2 \ldots) \in \Sigma_2$, then $\sigma$ is a homeomorphism and is called the \emph{shift map} on $\Sigma_2$.
\\Consider the non-autonomous system $(\Sigma_2, f_{1,\infty})$, where $f_{1, \infty}$  given by 
\begin{align} f_{1,\infty} :=  \{ I, I, \sigma, \sigma^{-1}, I, I, I, I, \sigma^{2}, \sigma^{-2}, \ldots,  \underbrace{I, I, \ldots, I}_{\text{$(2)^{r}$-times}},  \sigma^{r}, \sigma^{-r}, \ldots \}  \nonumber\end{align}
Let $\mathcal{F}$ be the Furstenberg family of all infinite subsets of $\mathbb{Z^+}$. For any $x\in \Sigma_2$, and any open neighbourhood $U$ of $x$, we will get a $y\in U$ such that $(x,y)$ is not an $\mathcal{F}$-asymptotic pair for $(\Sigma_2, f_{1,\infty})$.
Therefore, $(\Sigma_2, f_{1,\infty})$ is $\mathcal{F}$-sensitive.
\\Note that in the above example, the non-autonomous system $(\Sigma_2, f_{1,\infty})$ is not cofinite sensitive because for any open set $U\in \Sigma_2$, the compliment of the set $N_{f_{1,\infty}}(U)$ is not finite.\\ Also, this system is not syndetically sensitive because the set $N_{f_{1,\infty}}(U)$ has unbounded gaps, for any open set $U\in \Sigma_2$.
\\Thus, $\mathcal{F}$-sensitivity need not imply cofinite sensitivity and syndetic sensitivity.

\section{$\mathcal{F}$-sensitivity for different non-autonomous dynamical systems}
In this section,  we investigate $\mathcal{F}$-sensitivity and weakly $\mathcal{F}$-sensitivity for different non-autonomous systems. In the following theorem, we discuss the relation between the $\mathcal{F}$-sensitivity of a non-autonomous system $(X,f_{1,\infty})$ and its hyperspace $(\pazocal{K}(X),\overline{f_{1,\infty}})$.

\begin{thm}
Let $(X,f_{1,\infty})$ be a non-autonomous dynamical system and let $\mathit{k}\mathcal{F}$ be filterdual. Then, $(X,f_{1,\infty})$ is $\mathcal{F}$- sensitive if and only $(\pazocal{K}(X),\overline{f_{1,\infty}})$ is $\mathcal{F}$-sensitive. 
\end{thm}

\begin{proof}
Let $(\pazocal{K}$($X),\overline{f}_{1,\infty}$) be $\mathcal{F}$-sensitive with constant of $\mathcal{F}$-sensitivity $\delta>0$. For any $\epsilon>0$ and $x\in X$ let $U=B_d(x,\epsilon)$. Since $B_{d_H}(\{x\},\epsilon)$ is an $\epsilon$- neighbourhood of $\{x\}$ in $\pazocal{K}$($X)$ and $\overline{f_{1}^{\infty}}$ is $\mathcal{F}$-sensitive, therefore $N_{\overline{f_{1,\infty}}}[{B_{d_H}(\{x\},\epsilon),\delta)}]\in \mathcal{F}$ which implies there exist $A\in B_{d_H}(\{x\},\epsilon)$ and $n\geq 0$ such that $d_{H}(\overline{f_{1}^{n}}(\{x\}),\overline{f_{1}^{n}}(A))>\delta$.
\\Hence, we get $y\in A \subset U$ such that $d(f_{1}^{n}(x),f_{1}^{n}(y))>\delta$ which implies $N_{\overline{f_{1,\infty}}}$ $[\cup{B_{d_H}(\{x\},\epsilon),\delta)};x\in U] \subset N_{f_{1,\infty}}(U,\delta)$. As $N_{\overline{f_{1,\infty}}}[\cup{B_{d_H}(\{x\},\epsilon),\delta)};x\in U]\in \mathcal{F}$, therefore $N_{f_{1,\infty}}(U,\delta)\in \mathcal{F}$. Hence, $(X,f_{1,\infty})$ is $\mathcal{F}$-sensitive.

Conversely, assume that $(X,f_{1,\infty})$ is $\mathcal{F}$-sensitive with constant of $\mathcal{F}$-sensitivity $\delta>0$. Since $\pazocal{F}(X)$ is dense in $\pazocal{K}(X)$, it suffices to prove the result for $(\pazocal{F}(X),\overline{f|_{\pazocal{F}(X)}})$.

Let $A=\{x_1, x_2,\ldots, x_k\}\in \pazocal{F}(X)$ and $\epsilon>0$ be arbitrarily chosen. Then $\mathcal{F}$-sensitivity of $(X,f_{1,\infty})$ implies that the set $\{n\in \mathbb{N}$: diam$f_{1}^{n}(B_d(x_i,\epsilon/2))>\delta\}\in \mathcal{F}$, holds for each $i, 1\leq i\leq k$. Also we have $\mathit{k}\mathcal{F}$ is a filterdual, therefore \begin{align}
\Sigma = \bigcap_{i=1}^k\{n\in\mathbb{N}: diam f_{1}^{n}(B_d(x_i,\epsilon/2))>\delta\}\in \mathcal{F}. \nonumber
\end{align}
For any $n \in \Sigma$, there exists $(y_1,y_2,\ldots, y_k)\in B_d(x_1,\epsilon/2)\times \ldots \times B_d(x_k,\epsilon/2)$ such that $d(f_1^n(x_i),f_1^n(y_i))>\delta/2$ for each $i, 1\leq n\leq k$. Take $B=\{z_1,z_2,\ldots, z_k\}$ with \begin{enumerate}
\item $z_i=y_i, d(f_1^n(x_1),f_1^n(x_i))\leq \delta/2$,
\item $z_i=x_i$, otherwise.
\end{enumerate}

Then, \begin{align}
d_H(\overline{f_1^n}(A), \overline{f_1^n}(B))\geq min \{d(f_1^n(x_1),f_1^n(z_i)): 1\leq n\leq k\}>\delta/2.
\nonumber \end{align}
Note that $B\in \pazocal{F}(X)$ and $d_H(A,B)<\epsilon$. Therefore, \begin{align}
\Sigma \subset \{n\in \mathbb{N}: diam \overline{f_1^n}|_{\pazocal{F}(X)}(B_{d_{H}}(A,\epsilon)\cap \pazocal{F}(X))>\delta/2\}\in \mathcal{F}.
\nonumber \end{align}
Hence, $(\pazocal{F}(X), \overline{f_{1,\infty}}|_{\pazocal{F}(X)})$ is $\mathcal{F}$-sensitive. Therefore, $(\pazocal{K}(X),\overline{f_{1,\infty}})$ is $\mathcal{F}$-sensitive.
\end{proof}

\begin{rmk}
Note that the sufficiency part of the above theorem doesn't require $\mathcal{F}$ to be a filterdual. Therefore, $(\pazocal{K}(X),\overline{f_{1,\infty}})$ is $\mathcal{F}$-sensitive $\implies$ $(X,f_{1,\infty})$ is $\mathcal{F}$- sensitive in general.
\end{rmk}

\begin{thm}
Let $(X,f_{1,\infty})$ be a N.D.S generated by a family $f_{1,\infty}$ of feeble open maps commuting with $f$ and $\mathcal{F}$ be a filterdual and translation invariant Furstenberg family. If $\sum_{i=1}^\infty D(f_i,f) < \infty $, then $(X,f)$ is $\mathcal{F}$-sensitive if and only if $(X,f_{1,\infty})$ is $\mathcal{F}$-sensitive.
\end{thm}

\begin{proof}
Let $(X,f)$ be $\mathcal{F}$-sensitive with constant of $\mathcal{F}$-sensitivity $\delta > 0$. Let $\epsilon > 0$ be given and $U=B(x,\epsilon)$ be a non-empty open set in $X$. As $\sum_{i=1}^{\infty} D(f_n,f) < \infty$, by Lemma 2.2 we get that, for any $\epsilon > 0$, there exists $ n \in \mathbb{N} $ such that $d(f_{1}^{n+k}$ $(x),f^k(f_{1}^{n}$ $(x)))$ $ < \epsilon$, for every $x \in X$ and every $k\in \mathbb{N}$. Choose $m \in \mathbb{N}$ such that $1/m < \delta/4$. Then, there exists $n_0$ such that $d(f_{1}^{n_{0}+k}(x),f^k(f_{1}^{n_0}(x)))< 1/m$, for every $x\in X$ and every $k \in \mathbb{N}$. As $f_n's$ are feeble open, $U'=$int$f_{1}^{n_0}(U)$ is non-empty open and by $\mathcal{F}$-sensitivity of $(X,f)$, we get $N_{f}(U',\delta)\in \mathcal{F}$. Let $k \in N_{f}(U',\delta)$ then there exist $v_1$ and $v_2 \in f_{1}^{n_0}(U)$ such that $d(f^k(v_1),f^k(v_2))> \delta$. As $ v_1, v_2 \in f_{1}^{n_0}(U)$, there exist $v_{1}'$, $v_{2}'$ $\in U$ such that $v_1 = f_{1}^{n_0}(v_1'), v_2 = f_{1}^{n_0}(v_2')$ and \begin{align}d(f^k(f_{1}^{n_0}(v_1'),f^k(f_{1}^{n_0}(v_2'))) > \delta.\nonumber\end{align} Also $d(f_{1}^{n_{0}+k}(v_i'),f^k(f_{1}^{n_{0}}(v_i')))< 1/m$, for $i=1,2$. 

Therefore, by triangle inequality, \begin{align}d(f_{1}^{n_{0}+k}(v_1'),f_{1}^{n_{0}+k}(v_2'))> (\delta-2/m) > \delta/2.\nonumber\end{align} Thus, we obtain $v_1'$ and $v_2' \in U$ such that  \begin{align}d(f_{1}^{n_{0}+k}(v_1'),f_{1}^{n_{0}+k}(v_2'))> \delta/2\nonumber\end{align} which implies $n_{0}+k \in N_{f_{1,\infty}}(U,\delta/2)$ and hence $N_f(U',\delta) + n_{0} \subseteq N_{f_{1,\infty}}(U,\delta/2)$.
Since $N_f(U',\delta)\in \mathcal{F}$ and $\mathcal{F}$ is translation invariant therefore $N_{f_{1,\infty}}(U,\delta/2)\in \mathcal{F}$ and hence $(X,f_{1,\infty})$ is $\mathcal{F}$-sensitive.

Conversely, let $(X,f_{1,\infty})$ be $\mathcal{F}$-sensitive with constant $\delta>0$, $U$ be a non-empty open set in $X$ and $m \in \mathbb{N}$ be such that $1/m<\delta/4$. Then there exists $n_{0}\in\mathbb{N}$ such that $d(f_{1}^{n_{0}+k}(x),f^k(f_{1}^{n_{0}}(x)))<1/m$, for every $x\in X$ and every $k\in \mathbb{N}$. Since $(X,f_{1,\infty})$ is $\mathcal{F}$-sensitive, therefore for any $n \in \mathbb{N}$, the set \{$k:diam(f_{1}^{k}(U))>\delta$\} is infinite. Thus, for open set $(f_{1}^{n_0})^{-1}(U) $, the set $N_{f_{1,\infty}}((f_{1}^{n_0})^{-1}(U),\delta)\in \mathcal{F}$. So there exists $k \in \mathbb{N}$ such that $k+n_0 \in N_{f_{1,\infty}}((f_{1}^{n_0})^{-1}(U),\delta)$ implying there exist $v_1$ , $v_2$ $\in (f_{1}^{n_0})^{-1}(U)$ such that \begin{align}d(f_{1}^{n_{0}+k}(v_1),f_{1}^{n_{0}+k}(v_2))>\delta.\nonumber\end{align} Since $v_1$, $v_2$ $\in ((f_{1}^{n_0})^{-1}(U))$ therefore there exist $v_{1}'$ and $v_{2}'$ $\in U$ such that $v_{1}' = f_{1}^{n_0}(v_1)$ and $v_{2}'=f_{1}^{n_0}(v_2)$ and hence \begin{align}d((f_{n_{0}+k}\circ \ldots \circ f_{n_{0+1}})(v_{1}'),(f_{n_{0}+k}\circ \ldots \circ f_{n_{0+1}})(v_{2}')) > \delta.\nonumber\end{align} Also $d(f_{1}^{n_{0}+k}(v_i),f^k(f_{1}^{n_{0}}(v_i)))<1/m$ or $d(f_{1}^{n_{0}+k}(v_i),f^k(v_i'))<1/m$, for $i=1,2$.

Therefore, by triangle inequality \begin{align}d(f^k(v_{1}'),f^k(v_{2}'))>(\delta-2/m)>\delta/2.\nonumber\end{align} Hence, we obtain $v_{1}'$ and $v_{2}' \in U$ such that $d(f^k(v_{1}'),f^k(v_{2}'))>\delta/2$ which implies $k \in N_{f}(U,\delta/2)$ so $N_{f_{1,\infty}}((f_{1}^{n_0})^{-1}(U),\delta)- n_{0} \subseteq N_{f}(U,\delta/2)$ and $N_{f_{1,\infty}}((f_{1}^{n_0})^{-1}(U),\delta)\in \mathcal{F}$ implies $N_{f}(U,\delta/2)\in \mathcal{F}$.
Therefore, $(X,f)$ is $\mathcal{F}$-sensitive.
\end{proof}

\begin{thm}
Let $(X,f_{1,\infty})$ be a N.D.S. which converges uniformly to a map $f$. If $(X,f_{1,\infty})$ is $\mathcal{F}$-sensitive, then for any $k\geq 1$, $(X,f_{1,\infty}^{[k]})$ is also $\mathcal{F}$-sensitive.
\end{thm}

\begin{proof}
For any fixed $k\geq 2$, as $(X,f_{1,\infty})$ is $\mathcal{F}$-sensitive with constant of $\mathcal{F}$-sensitivity $\delta>0$, therefore, for any $x\in X$ and any $\epsilon>0$, there exists $y\in B_d(x,\epsilon)$ such that $(x,y)\notin Asym_{\epsilon}(\mathcal{F})$. Hence, there exists $n_{x}(\epsilon)\in \mathbb{N}$ such that $d(f_{1}^{n_x(\epsilon)}(x),f_1^{n_x(\epsilon)}(y))>\delta$. Now, by Lemma 2.1, we have that for any given $\delta>0$, there exist $\eta>0$ and $N_0\geq 3k$ such that for any pair $(x,y)\in X$ with $d(x,y)<\eta$, we have \begin{align*}
& d(f_n^i(x),f_n^i(y))<\delta,\hspace{0.2cm}for\hspace{0.2cm} all\hspace{0.2cm} n\geq N_0\hspace{0.2cm} and\hspace{0.2cm} 0\leq i\leq k \\ \iff & d(f_n^i(x),f_n^i(y))\geq \delta \implies d(x,y)\geq \eta,\hspace{0.2cm} for \hspace{0.2cm} all\hspace{0.2cm} n\geq N_0 \hspace{0.2cm}and\hspace{0.2cm} 0\leq i\leq k. 
\end{align*}
Now, for any $1\leq i\leq 2N_0$, $f_1^i$ is uniformly continuous, therefore for given $\delta>0$, there exists $\epsilon'>0$ such that \begin{align}
d(x_1,x_2)< \epsilon' \implies d(f_1^i(x),f_1^i(y))<\delta,\hspace{0.2cm} for\hspace{0.2cm} any\hspace{0.2cm} x_1, x_2 \in X \hspace{0.2cm}and\hspace{0.2cm} 1\leq i\leq 2N_0.
\end{align}
Hence, for any $x\in X$ and any $\epsilon< \epsilon'$, we have $n_x(\epsilon)> 2N_0\geq 6k$. By division algorithm, there exists $i_0\in \{0,1,2,\ldots, n-1\}$ such that $n_x(\epsilon)=kl+i_0$, for some $l\in \mathbb{N}$ which implies $n_x(\epsilon)-i_0=kl$. Combining this with the fact that for any $\epsilon< \epsilon'$, $d(f_1^{n_x(\epsilon)}(x),f_1^{n_x(\epsilon)}(y))$ = $d(f_{n_x(\epsilon)-i_0+1}^{i_0}(f_1^{n_x(\epsilon)-i_0}(x)),f_{n_x(\epsilon)-i_0+1}^{i_0}(f_1^{n_x(\epsilon)-i_0}(y)))>\delta$. Also, since $n_x(\epsilon)>2N_0\geq 6k$ and $0\leq i_0< k$, therefore $n_x(\epsilon)-i_0+1\geq N_0$.\\ Using (1), we get that for any $\epsilon< \epsilon'$, $d(f_1^{n_x(\epsilon)-i_0}(x),f_1^{n_x(\epsilon)-i_0}(y))\geq\eta.$\\ Now since $k$ divides $n_x(\epsilon)-i_0$, therefore we have $d(f_{k(l-1)+1}^k\circ \ldots \circ f_1^k(x),f_{k(l-1)+1}^k\circ \ldots \circ f_1^k(y))\geq$ $d(f_1^{n_x(\epsilon)-i_0}(x),f_1^{n_x(\epsilon)-i_0}(y))$, and hence we get $d(f_{k(l-1)+1}^k\circ \ldots \circ f_1^k(x),f_{k(l-1)+1}^k\circ \ldots \circ f_1^k(y))\geq\eta$. Thus, $(x,y)$ is not a $\mathcal{F}$-$\delta_0$-asymptotic pair for $(X,f_{1,\infty}^{[k]})$, $0< \delta_{0} < \eta$ implying $(X,f_{1,\infty}^{[k]})$ is $\mathcal{F}$- sensitive, with $\mathcal{F}$-sensitivity constant $0< \delta_{0} < \eta$.
\end{proof}

\begin{rmk}
Converse of the above result is also true. If $(X,f_{1,\infty}^{[k]})$ is $\mathcal{F}$-sensitive, for any $k\in \mathbb{N}$, then in particular for $n=1$, we get that $f_{1,\infty}^{[1]}=(f_{(n-1)+1}^1)_{n=1}^{\infty}$ which implies $(X,f_{1,\infty})$ is $\mathcal{F}$-sensitive.
\end{rmk}

\begin{rmk}
Based on similar arguments, one can prove the above result for $(X,f_{1,\infty})$ being weakly $\mathcal{F}$-sensitive.
\end{rmk}

In \cite{MR3661658}, the authors have investigated various dynamical properties of a non-autonomous system generated by a finite family of maps. We explore $\mathcal{F}$-sensitivity for such non-autonomous systems.
\begin{thm}
Let $(X,f_{1,\infty})$ be a non-autonomous system generated by a finite family of maps, $\mathbb{F}=\{f_1, f_2, \ldots, f_k\}$. If the autonomous system $(X,f_k\circ f_{k-1} \circ \ldots \circ f_1)$ is $\mathcal{F}$-sensitive, then $(X,f_{1,\infty})$ is also $\mathcal{F}$-sensitive. 
\end{thm}

\begin{proof}
Let $x\in X$ and $U$ be any open neighbourhood of $x$. Suppose $(X,f_k, f_{k-1}\circ f_1)$ is $\mathcal{F}$-sensitive with $\mathcal{F}$-sensitivity constant $\delta>0$. By $\mathcal{F}$-sensitivity of $(X,f_k\circ f_{k-1}\circ \ldots \circ f_1)$, we get a $y\in U$ such that the set $\{n\in \mathbb{N}:d(({f_1^k})^n(x),({f_1^k})^n)(y))>\delta\} \in \mathcal{F}$. Therefore, the set $\{m\in \mathbb{N}:d(f_1^m(x),f_1^m(y))>\delta\}\in \mathcal{F}$. Hence, $(X,f_{1,\infty})$ is $\mathcal{F}$-sensitive.
\end{proof}

\begin{rmk}
A similar result, as Theorem 3.6, holds true for weakly $\mathcal{F}$-sensitivity. 
\end{rmk}

In the following example, we show that the non-autonomous system generated by the family $\mathbb{F}$ can be $\mathcal{F}$-sensitive, even when none of the maps in the family $\mathbb{F}$ is $\mathcal{F}$-sensitive.

\begin{exm} \end{exm}
Let $I$ be the interval $[0,1]$, $\mathcal{F}$ be the Furstenberg family of all infinite subsets of $\mathbb{Z^+}$ and $f_1, f_2$ on $I$ be defined by: 
\[ f_1(x) = \begin{cases}
4x,  & \text{for} \ x \in \left[0, \frac{1}{4}\right] \\

5/4-x, & \text{for} \ x \in \left[\frac{1}{4},1\right].
\end{cases} \]

\[ f_2(x) = \begin{cases}
1/4-x,  & \text{for} \ x \in \left[0, \frac{1}{4}\right] \\
4x-1, & \text{for} \ x \in \left[\frac{1}{4},\frac{1}{2}\right]\\
2-2x, & \text{for} \ x \in \left[\frac{1}{2},1\right].
\end{cases} \]

Let $\mathbb{F}=\{f_1,f_2\}$ and $(X,f_{1,\infty})$ be non-autonomous system generated by $\mathbb{F}$. Note that [1/4,1] is an invariant set for $f_1$ and [0,1/4] is an invariant set for $f_2$. Therefore, neither $f_1$ nor $f_2$ is $\mathcal{F}$-sensitive. However, their composition given by
 
\[ f_2\circ f_1(x) = \begin{cases}
1/4-4x,  & \text{for} \ x \in \left[0, \frac{1}{16}\right] \\
16x-1, & \text{for} \ x \in \left[\frac{1}{16},\frac{1}{8}\right]\\
2-8x, & \text{for} \ x \in \left[\frac{1}{8},\frac{1}{4}\right]\\
2x-1/2, & \text{for} \ x \in \left[\frac{1}{4},\frac{3}{4}\right]\\
4-4x, & \text{for} \ x \in \left[\frac{3}{4},1\right].
\end{cases} \]
is $\mathcal{F}$-sensitive and hence by Theorem 3.6, the non-autonomous system $(X,f_{1\infty})$ is $\mathcal{F}$-sensitive.

\section*{Acknowledgement}
The first author is funded by GOVERNMENT OF INDIA, MINISTRY OF SCIENCE and TECHNOLOGY No: DST/INSPIRE Fellowship/[IF160750].


\begin{thebibliography}{99}

\bibitem{MR1269778} G. Beer, {\it Topologies on closed and closed convex sets, vol. 268 of Mathematics and
its Applications}, Kluwer Academic Publishers Group, Springer-Verlag, Dordrecht.

\bibitem{MR2} X.-F. Dinga, T.-X. Lub, and J.-J. Wangc, {\it Sensitivity of non-autonomous dis-
crete dynamical systems revisited}, J. Nonlinear Sci. Appl., 10 (2017), pp. 5239-5244.

\bibitem{MR2960260} Dvo\v r\'akov\'a, J. , {\it Chaos in nonautonomous discrete dynamical systems}, Commun. Nonlinear Sci.
Numer. Simul., 17 (2012), pp. 4649-4652.

\bibitem{MR2922208} R. Li, {\it A note on stronger forms of sensitivity for dynamical systems}, Chaos Solitons
Fractals, 45 (2012), pp. 753-758.

\bibitem{MR3623040} R. Li, Y. Zhao, and H. Wang, {\it Furstenberg families and chaos on uniform limit maps}, J. Nonlinear Sci. Appl., 10 (2017), pp. 805-816.

\bibitem{MR3344142} L. Liu and Y. Sun, {\it Weakly mixing sets and transitive sets for non-autonomous discrete systems}, Adv. Difference Equ., (2014), pp. 217, 9.

\bibitem{MR3528201} T. Lu and G. Chen, {\it Proximal and syndetical properties in nonautonomous discrete
systems}, J. Appl. Anal. Comput., 7 (2017), pp. 92-101.

\bibitem{MR3584037} C. Ma, P. Zhu, and R. Li, {\it On iteration invariants for {$(\mathcal {F}_1,\mathcal{F}_2)$}-sensitivity and weak {$(\mathcal {F}_1,\mathcal{F}_2)$}-sensitivity of non-autonomous discrete systems}, J. Nonlinear Sci. Appl., 9 (2016), pp. 5772-5779.

\bibitem{MR3206430} R. Memarbashi and H. Rasuli, {\it Notes on the dynamics of nonautonomous dis-
crete dynamical systems}, J. Adv. Res. Dyn. Control Syst., 6 (2014), pp. 8-17.

\bibitem{MR2351026} T. K. S. Moothathu,, {\it Stronger forms of sensitivity for dynamical systems, Non-
linearity}, 20 (2007), pp. 2115-2126.

\bibitem{MR3584171} I. Sanchez, M. Sanchis, and H. Villanueva, {\it Chaos in hyperspaces of non-autonomous discrete systems}, Chaos Solitons Fractals, 94 (2017), pp. 68-74.

\bibitem{MR2665229} P. Sharma and A. Nagar, {\it Inducing sensitivity on hyperspaces}, Topology Appl.,
157 (2010), pp. 2052-2058.

\bibitem{MR3661658} P. Sharma and M. Raghav, {\it Dynamics of non-autonomous discrete dynamical systems}, Topology Proc., 52 (2018), pp. 45-59.

\bibitem{MR3485444} M. Stefankova, {\it Inheriting of chaos in uniformly convergent nonautonomous dynamical systems on the interval}, Discrete Contin. Dyn. Syst., 36 (2016), pp. 3435-3443.

\bibitem{MR2600324} H. Wang, J. Xiong, and F. Tan, {\it Furstenberg families and sensitivity}, Discrete Dyn. Nat. Soc.,(2010), pp. Art. ID 649348, 12.

\bibitem{MR3543466} X. Wu, R. Li, and Y. Zhang, {\it The multi-{$\mathcal {F}$}-sensitivity and {$(\mathcal{F}_1,\mathcal{F}_2)$}-sensitivity for product systems}, J. Nonlinear Sci. Appl., 9 (2016), pp. 4364-4370.

\bibitem{MR3339062} X. Wu, J. Wang, and G. Chen, {\it F-sensitivity and multi-sensitivity of hyperspatial
dynamical systems}, J. Math. Anal. Appl., 429 (2015), pp. 16-26.


\bibitem{MR3019971} X. Wu and P. Zhu, {\it Chaos in a class of non-autonomous discrete systems}, Appl. Math. Lett., 26 (2013), pp. 431-436.

\end{thebibliography}
\end{document}